\documentclass[12pt]{article}
\usepackage{a4wide}
\usepackage{graphicx}
\usepackage{amsthm}
\usepackage{amssymb}
\usepackage{todonotes}

\newtheorem{corollary}{Corollary}
\newtheorem{theorem}{Theorem}
\newtheorem{remark}{Remark}
\theoremstyle{definition}

\def\N{{{\rm I}\!{\rm N}}}
\def\R{{{\rm I}\!{\rm R}}}

\def\be#1{\begin{equation} \label{#1} }
\def\ee{\end{equation}}
\def\barr{\begin{array}}
\def\earr{\end{array}}
\def\eqref#1{(\ref{#1})}

\def\tfrac#1#2{{\textstyle \frac{#1}{#2}}}
\def\xdag{x^\dagger}
\def\ydel{y^\delta}
\def\cD{\mathcal{D}}
\def\cR{\mathcal{R}}
\def\Xad{X^{\mbox{\footnotesize ad}}}
\def\x1{x_{k+1}^\delta}
\def\xk{x_k^\delta}
\def\xh1{x_{k+1,h}^\delta}
\def\xkh{x_{k,h}^\delta}
\def\src{s}
\begin{document}
%
%\title[Convergence of the IRGNM under a tangential cone condition]{Convergence of the IRGNM Tikhonov and the IRGNM Ivanov method under a tangential cone condition in Banach space}
\title{Convergence and adaptive discretization of the IRGNM Tikhonov and the IRGNM Ivanov method under a tangential cone condition in Banach space
}
\author{Barbara Kaltenbacher and Mario Luiz Previatti de Souza}

%\address{Institute of Mathematics, Alpen-Adria-Universit\"at Klagenfurt}
%\ead{barbara.kaltenbacher@aau.at, mario.previatti@aau.at}

\maketitle

\begin{abstract}
In this paper we consider the Iteratively Regularized Gauss-Newton Method (IRGNM) in its classical Tikhonov version and in an Ivanov type version, where regularization is achieved by imposing bounds on the solution. We do so in a general Banach space setting and under a tangential cone condition, while convergence (without source conditions, thus without rates) has so far only been proven under stronger restrictions on the nonlinearity of the operator and/or on the spaces.
Moreover, we provide a convergence result for the discretized problem with an appropriate control on the error and show how to provide the required error bounds by goal oriented weighted dual residual estimators. 
The results are illustrated for an inverse source problem for a nonlinear elliptic boundary value problem, for the cases of a measure valued and of an $L^\infty$ source. For the latter, we also provide numerical results with the Ivanov type IRGNM. 
\end{abstract}

\section{Introduction}
In this paper we consider a nonlinear ill-posed operator equation
\begin{equation}\label{Fxy}
F(x)=y\,,
\end{equation}
where the possibly nonlinear operator $F:\cD (F)\subseteq X \to
Y$ with domain $\cD (F)$ maps between real Banach spaces $X$ and $Y$.
We are interested in the ill-posed situation, i.e., $F$ fails
to be continuously invertible, and the data are contaminated with
noise, thus regularization has to be applied (see, e.g.,
\cite{EHNBuch,TikAr77}, and references therein).

Throughout this paper we will assume that an exact solution
$x^\dagger  \in \cD (F)$ of \eqref{Fxy} exists, i.e.,
$F(x^\dagger)=y$, and that the noise level $\delta$
in the (deterministic) estimate
\begin{equation}\label{delta}
\|y-y^\delta\|\leq\delta
\end{equation}
is known.

Partially we will also refer to the formulation of the inverse problem as a system of model and observation equation
\begin{eqnarray}
A(x,u)&=&0	\label{mod}\\
C(u)&=&y\,.	\label{obs}
\end{eqnarray}
Here $A:X\times V\to W^*$ and $C:V\to Y$ are the model and observation operator, so that with the parameter-to-state map $S:X\to V$ satisfying $A(x,S(x))=0$ and $F=C\circ S$, \eqref{Fxy} is equivalent to the all-at-once formulation \eqref{mod}, \eqref{obs}.

\medskip

Newton type methods for the solution of nonlinear ill-posed problems \eqref{Fxy} have been extensively studied in Hilbert spaces (see, e.g., \cite{BakuKokurin,KNSBuch} and the references therein) and more recently also in a in Banach space setting. In particular, the iteratively regularized Gauss-Newton method \cite{Baku92} can be generalized to a Banach space setting by calculating iterates $x_{k+1}^{\delta}$ in a Tikhonov type variational form as
\begin{equation}\label{IRGNMTikhonov}
\x1 \in{\rm argmin}_{x\in \cD (F)} \ \|F'(x_k^\delta)(x-x_k^\delta)+F(x_k^\delta)-\ydel\|^p +\alpha_k\cR(x)\,,
\end{equation}
see, e.g., \cite{HohageWerner13,JinZhong13,KH10,KSS09,Werner14}
where $p\in[1,\infty)$, $(\alpha_k)_{k\in\N}$ is a sequence of regularization parameters, $\cR$ is some regularization functional, and the prime denotes the G\^ateaux derivative.
Alternatively, one might introduce regularization by imposing some bound $\rho_k$ on the norm of $x$, or, again, generally, on a regularization functional of $x$
\begin{equation}\label{IRGNMIvanov}
x_{k+1}^{\delta} \in{\rm argmin}_{x\in \cD (F)} \ \frac{1}{2}\|F'(x_k^\delta)(x-x_k^\delta)+F(x_k^\delta)-\ydel\|^2 \mbox{ such that }\cR(x)\leq\rho_k\,,
\end{equation}
which corresponds to Ivanov regularization or the method of quasi solutions, see, e.g., \cite{DombrovskajaIvanov65, Ivanov62, Ivanov63, IvanovVasinTanana02, LorenzWorliczek13, NeubauerRamlau14, SeidmanVogel89}.
We restrict ourselves to the norm in $Y$ as a measure of the data misfit, but the analysis could as well be extended to more general functionals $\mathcal{S}$ satisfying certain conditions, as e.g., in \cite{HohageWerner13,Werner14}. 

As a constraint on the nonlinearity of the forward operator $F$ we impose the tangential cone condition
\begin{equation}\label{tangcone}
\|F(\tilde{x})-F(x)-F'(x)(\tilde{x}-x)\|\leq c_{tc}\|F(\tilde{x})-F(x)\|
\mbox{ for all }x\in \mathcal{B}_R
\end{equation}
(also called Scherzer condition, cf. \cite{Scherzer95})
for some set $\mathcal{B}_R\subseteq \mathcal{D}(F)\not=\emptyset$ and $c_{tc}<1/3$. Note that the  convergence conditions imposed in \cite{HohageWerner13,JinZhong13,KSS09,KH10,Werner14} in the situation without source condition, namely local invariance of the range of $F'(x)^*$, are slightly stronger, since this adjoint range invariance is sufficient for \eqref{tangcone}. However, most probably the gap is not very large, as in those application examples where \eqref{tangcone} has been verified, the proof of \eqref{tangcone} is actually done via adjoint range invariance.  

The remainder of this paper is organized as follows. In Section \ref{sec:conv} we state and prove convergence results in the continuous and discretized setting. 
Section \ref{sec:errest} shows how to actually obtain the required discretization error estimates by a goal oriented weighted dual residual approach and Section \ref{sec:modex} illustrates the theoretical finding by an inverse souce problem for a nonlinear PDE.
In Section \ref{sec:num} we provide some numerical results for this model problem and Section \ref{sec:conclrem} concludes with some remarks.

\section{Convergence}\label{sec:conv}
In this section we will study convergence of the IRGNM iterates first of all in a continuous setting, then in the situation of having discreted for computational purposes. 

The regularization parameters $\alpha_k$ and $\rho_k$ are chosen a priori 
\begin{equation}\label{alphak}
\alpha_k=\alpha_0 \theta^k \mbox{ for some } \theta\in((\tfrac{2c_{ct}}{1-c_{ct}})^p,1)
\end{equation}
(note that $(\tfrac{2c_{ct}}{1-c_{ct}})^p<1$ for $c_{tc}<1/3$) and 
\begin{equation}\label{rhok}
\rho_k\equiv \rho\geq \cR(\xdag)\,, 
\end{equation}
and the iteration is stopped according to the discrepancy principle 
\begin{equation}\label{discrprinc}
k_*=k_*(\delta,\ydel)=\min\{k\in\N_0\ : \ \|F(x_k^{\delta})-\ydel\|\leq\tau\delta\}
\end{equation}
with some fixed $\tau>1$ chosen sufficiently large but independent of $\delta$.

\begin{theorem} \label{th:conv}
Let $\cR$ be proper, convex and lower semicontintuous with $\cR(x^\dagger)<\infty$ and let, for all $r\geq \cR(x^\dagger)$, the sublevel set
\[
\mathcal{B}_r =\{ x\in \mathcal{D}(F)\, : \, \cR(x)\leq r\}
\]
be compact with respect to some topology $\mathcal{T}$ on $X$.
\\
Moroever, let $F$ be G\^ateaux differentiable in $\mathcal{B}_R$, satisfy \eqref{tangcone}, and let, for all $x\in\mathcal{B}_R$, $F'(x)$ and $F$ be $\mathcal{T}$-to-norm continuous, for some appropriately chosen $R>\cR(\xdag)$. Finally, let the family of data $(\ydel)_{\delta>0}$ satisfy \eqref{delta}.

\smallskip

\begin{enumerate}
\item[(i)] 
Then for fixed $\delta$, $\ydel$, the iterates according to \eqref{IRGNMTikhonov} and \eqref{IRGNMIvanov} are well-defined and remain in $\mathcal{B}_R$, and the stopping index $k_*$ according to the discrepancy principle is finite.
\item[(ii)] 
Moreover, for both methods we have $\mathcal{T}$-subsequential convergence as $\delta\to0$ i.e.,\\ $(x^\delta_{k_*(\delta,\ydel)})_{\delta>0}$ has a $\mathcal{T}$-convergent subsequence and the limit of every $\mathcal{T}$-convergent subsequence solves \eqref{Fxy}. If the solution $\xdag$ of \eqref{Fxy} is unique in $\mathcal{B}_R$, then $x^\delta_{k_*(\delta,\ydel)}\stackrel{\mathcal{T}}{\longrightarrow}\xdag$ as $\delta\to0$.
\item[(iii)] 
Additionally, $k_*$ satisfies the asymptotics $k_*=\mathcal{O}(\log(1/\delta))$.
\end{enumerate}
\end{theorem}
\begin{proof}
Existence of minimizers $x_{k+1}^\delta$ of \eqref{IRGNMTikhonov} and \eqref{IRGNMIvanov} for fixed $k$, $x_k^\delta$ and $\ydel$ follows by the direct method of calculus of variations: In both cases, the cost functional 
\begin{eqnarray*}
J_k(x)=\|F'(x_k^\delta)(x-x_k^\delta)+F(x_k^\delta)-\ydel\|^p+\alpha_k\cR(x)
\mbox{ in case of \eqref{IRGNMTikhonov}},\\
J_k(x)=\frac{1}{2}\|F'(x_k^\delta)(x-x_k^\delta)+F(x_k^\delta)-\ydel\|^2
\mbox{ in case of \eqref{IRGNMIvanov}},
\end{eqnarray*}
is bounded from below and the admissible set 
\[
\Xad=\mathcal{D}(F)\mbox{ in case of \eqref{IRGNMTikhonov}}, \quad
\Xad=\mathcal{D}(F)\cap\mathcal{B}_{\rho_k}\mbox{ in case of \eqref{IRGNMIvanov}},
\]
is nonempty (for \eqref{IRGNMIvanov} this follows from $\rho_k\geq\cR(\xdag)$). Hence, there exists a minimizing sequence $(x^l)_{l\in\N}\subseteq \Xad\cap\mathcal{B}_r$ for 
\[
r=\frac{1}{\alpha_k}J_k(\xdag)
\mbox{ in case of \eqref{IRGNMTikhonov}}, \quad
r=\rho_k
\mbox{ in case of \eqref{IRGNMIvanov}},
\]
with $\lim_{l\to\infty}J_k(x^l)=\inf_{x\in \Xad} J_k(x)$.

By $\mathcal{T}$-compactness of $\Xad\cap\mathcal{B}_r=\mathcal{B}_r$, the sequence $(x^l)_{l\in\N}$ has a $\mathcal{T}$-convergent subsequence $(x^{l_m})_{m\in\N}$ with limit $\bar{x}\in \Xad\cap\mathcal{B}_r$.
Since $F'(x_k^\delta)$ is $\mathcal{T}$-to-norm continuous, we also have $\mathcal{T}$-continuity of $x\mapsto\|F'(x_k^\delta)(x-x_k^\delta)+F(x_k^\delta)-\ydel\|$, hence $\mathcal{T}$-lower semicontinuity of $J_k$, (which, in case of \eqref{IRGNMTikhonov}, is the sum of a $\mathcal{T}$-continuous and a $\mathcal{T}$-lower semicontinuous function). Thus altogether $J_k(\bar{x})\leq\liminf_{m\to\infty}J_k(x^{l_m})=\inf_{x\in \Xad} J_k(x)$ and $\bar{x}\in \Xad$, hence $\bar{x}$ is a minimizer.

\medskip

Note that (ii) follows from (i) by standard arguments and our assumption on $\mathcal{T}$-compactness of $\mathcal{B}_R$. Thus it remains to  prove (i) and (iii)for the two versions \eqref{IRGNMTikhonov}, \eqref{IRGNMIvanov} of the IRGNM.

For this purpose we are going to show that for every $\delta >0$, there exists $k_*=k_*(\delta,\ydel)$ such that $k_* \sim \log (1/\delta)$, and the stopping criterion according to the discrepancy principle $\|F(x^\delta_{k_*(\delta,\ydel)})-\ydel\| \leq \tau\delta$ is satisfied. For \eqref{IRGNMTikhonov}, we also need to show that $\cR (x^\delta_{k_*(\delta,\ydel)})$ is bounded, whereas in \eqref{IRGNMIvanov} this automatically holds by \eqref{rhok}.

We start with \eqref{IRGNMTikhonov}. 
Using the minimality of $\x1$ and \eqref{delta}, \eqref{tangcone}, as well as $\xdag\in \cD(F)$, we have 
\begin{eqnarray*}
&&\|F'(\xk)(\x1-\xk)+F(\xk)-\ydel\|^p +\alpha_k\mathcal{R}(\x1) \\
&&\leq \|F'(\xk)(\xdag-\xk)+F(\xk)-\ydel\|^p + \alpha_k\mathcal{R}(\xdag)\\
&&\leq \Bigl(c_{tc}\|F(\xk)-\ydel\|+(1+c_{tc})\delta\Bigr)^p + \alpha_k\mathcal{R}(\xdag),
\end{eqnarray*}
and on the other hand
\begin{eqnarray*}
&&\|F'(\xk)(\x1-\xk)+F(\xk)-\ydel\|^p +\alpha_k\mathcal{R}(\x1) \\
&&\geq \Bigl((1-c_{tc})\|F(\x1)-\ydel\|-c_{tc}\|F(\xk)-\ydel\|\Bigr)^p +\alpha_k\mathcal{R}(\x1).
\end{eqnarray*}

To handle the power $p$ we make use of the following inequalities that can be proven by solving	extremal value problems, see the appendix
\begin{equation}\label{abab}
(a+b)^p \leq (1+\gamma)^{p-1}a^p+\left(\frac{1+\gamma}{\gamma}\right)^{p-1}b^p \mbox{ and } (a-b)^p \geq (1-\epsilon)^{p-1}a^p-\left(\frac{1-\epsilon}{\epsilon}\right)^{p-1}b^p,
\end{equation}
for all $a,b > 0,$ $p \geq 1$ and $\gamma, \epsilon \in (0,1)$.

Hence, the following general estimate holds 
\begin{eqnarray}\label{genfor}
&&(1-\epsilon)^{p-1}(1-c_{tc})^p\|F(\x1)-\ydel\|^p+\alpha_k\mathcal{R}(\x1)\\
&&\leq \left( (1+\gamma)^{p-1}+\left(\frac{1-\epsilon}{\epsilon}\right)^{p-1}\right)c_{tc}^p\|F(\xk)-\ydel\|^p+\alpha_k\mathcal{R}(\xdag)+\left(\frac{1+\gamma}{\gamma}\right)^{p-1}(1+c_{tc})^p\delta^p,\nonumber
\end{eqnarray}
for $\gamma, \epsilon \in (0,1).$

So in order for this recursion to yield geometric decay of $\|F(\xk)-\ydel\|$, we need to ensure 
\begin{equation}\label{q1}
(1-\epsilon)^{p-1}(1-c_{tc})^p > \left( (1+\gamma)^{p-1}+\left(\frac{1-\epsilon}{\epsilon}\right)^{p-1}\right)c_{tc}^p 
\end{equation}
for a proper choice of $\epsilon,\gamma\in(0,1)$. To obtain the largest possible (and therefore least restrictive) bound on $c_{tc}$, we rewrite the requirement above as
\begin{eqnarray*}
\left(\frac{c_{tc}}{1-c_{tc}}\right)^p &<& 
\sup_{\epsilon,\gamma\in(0,1)} (1-\epsilon)^{p-1}\left( (1+\gamma)^{p-1}+\left(\frac{1-\epsilon}{\epsilon}\right)^{p-1}\right)^{-1}\\
&=&\sup_{\epsilon\in(0,1)} \underbrace{(1-\epsilon)^{p-1}\left( 1+\left(\frac{1-\epsilon}{\epsilon}\right)^{p-1}\right)^{-1}}_{=\phi(\epsilon)} = \phi(\tfrac12)=2^{-p},
\end{eqnarray*}
as can be found out by evaluating the derivative of $\phi$
\[
\phi'(\epsilon)= -(p-1)(1-\epsilon)^{p-2}\left( 1+\left(\frac{1-\epsilon}{\epsilon}\right)^{p-1}\right)^{-2}\left(1-\left(\frac{1-\epsilon}{\epsilon}\right)^p\right)\,.
\]
Thus we will furtheron set $\epsilon=\frac12$ and assume that $\gamma>0$ is sufficiently small so that \eqref{q1} holds with $\epsilon=\frac12$, i.e., 
\begin{equation}\label{q}
q:=\frac{(1+\gamma)^{p-1}+1}{2}\left(\frac{2c_{tc}}{1-c_{tc}}\right)^{p} \in (0,1)\,.
\end{equation}
Additionally, we use the following abbreviations 
\begin{eqnarray*}
&&d_k:= 2^{1-p}(1-c_{tc})^p\|F(\xk)-\ydel\|^p,\\
&&\mathcal{R}_k:=\mathcal{R}(\xk),  \quad \mathcal{R}^\dagger:=\mathcal{R}(\xdag),\\
&&C:=\left(\frac{1+\gamma}{\gamma}\right)^{p-1}(1+c_{tc})^p.
\end{eqnarray*}
Then, using \eqref{alphak}, estimate \eqref{genfor} can be written as
\begin{equation}\label{recursion}
d_{k+1}+\alpha_k\mathcal{R}_{k+1}\leq qd_k +\alpha_0\theta^k\mathcal{R}^\dagger+C\delta^p,
\end{equation}
which we first of all regard as a recursive estimate for $d_k$.

By induction, we have for all $l\in\{0,\ldots,k\}$
\begin{equation}\label{induction}
d_{k+1}+\alpha_k\mathcal{R}_{k+1} \leq q^{l+1} d_{k-l}+\left(1+\frac{q}{\theta}+\ldots+\left(\frac{q}{\theta}\right)^l\right)\alpha_0 \theta^k\mathcal{R}^\dagger +(1+q+\ldots+q^l)C\delta^p.
\end{equation}
Indeed, for $l=0$, \eqref{induction} is just \eqref{recursion}.
Suppose that \eqref{induction} holds for $l$, then using \eqref{recursion} with $k$ replaced by $k-(l+1)$, we obtain the formula for $l+1$
\begin{eqnarray*}
&&d_{k+1}+\alpha_k\mathcal{R}_{k+1} \\
&&\leq q^{l+1}\Bigl(qd_{k-(l+1)}+\alpha_0 \theta^{k-(l+1)}\mathcal{R}^\dagger +C\delta^p\Bigr)\\
&&\qquad+\left(1+\frac{q}{\theta}+\ldots+\left(\frac{q}{\theta}\right)^l\right)\alpha_0 \theta^k\mathcal{R}^\dagger +(1+q+\ldots+q^l)C\delta^p\\ 
&&= q^{(l+1)+1} d_{k-(l+1)}+\left(1+\frac{q}{\theta}+\ldots+\left(\frac{q}{\theta}\right)^l+\left(\frac{q}{\theta}\right)^{l+1}\right)\alpha_0 \theta^k\mathcal{R}^\dagger \\
&&\qquad+(1+q+\ldots+q^l+q^{l+1})C\delta^p,
\end{eqnarray*}
and the induction proof is complete.

Hence, setting $l=k$ in \eqref{induction} and using the geometric series formula, we get 
\begin{eqnarray}\label{indu}
d_{k+1}+\alpha_k\mathcal{R}_{k+1} < q^{k+1}d_0 +\left(\frac{1}{1-\frac{q}{\theta}}\right)\alpha_k\mathcal{R}^\dagger+\left(\frac{1}{1-q}\right)C\delta^p.
\end{eqnarray}
provided $\frac{q}{\theta}<1$, which by definition of $q$ \eqref{q} is achievable for $\gamma>0$ sufficiently small, due to $\theta>(\tfrac{2c_{ct}}{1-c_{ct}})^p$, cf. \eqref{alphak}.

We next show that the discrepancy stopping criterion from \eqref{discrprinc}, i.e.,  $d_{k_*} \leq \tilde{\tau}\delta^p$ for $\tilde{\tau}=2^{1-p}(1-c_{tc})^p\tau^p$, will be satisfied after finitely many, namely $O(\log(1/\delta))$, steps.
For  this  purpose, note that $\tilde{\tau}>\frac{C}{1-q}$, provided $\tau$ is chosen sufficiently large, which we assume to be done. Thus, indeed, using \eqref{alphak}, \eqref{indu}, we have
\begin{equation}\label{idk}
d_{k}\leq d_{k}+\alpha_{k-1}\mathcal{R}_{k} < \theta^{k}\Bigl(d_0 +\frac{\alpha_0}{\theta -q} \mathcal{R}^\dagger\Bigr)+\frac{C}{1-q}\delta^p,
\end{equation}
where the right hand side falls below $\tilde{\tau}\delta^p$ as soon as
\[
k\geq (\log 1/\theta)^{-1}\left(p\log (1/\delta) +\log \left(d_0+\frac{\alpha_0}{\theta -q}\mathcal{R}^\dagger\right)-\log \left(\tilde{\tau}-\frac{C}{1-q}\right)\right)=:\bar{k}(\delta).
\]
Thus we get the upper estimate $k_{*}(\delta,\ydel)\leq\bar{k}(\delta)=O(\log(1/\delta))$.

To finish the convergence proof of \eqref{IRGNMTikhonov} we estimate $\mathcal{R}(x_{k_*(\delta,\ydel)}^{\delta})$. According to our notation, from \eqref{idk} and the fact that  $\alpha_{k-1}\mathcal{R}_{k}\leq d_{k}+\alpha_{k-1}\mathcal{R}_{k}$ as well as the bound on $k_*(\delta,\ydel)$ we just derived, we get, for all $k\in\{1,\ldots, k_*(\delta,\ydel)\}$
\begin{eqnarray}
\mathcal{R}_{k} &\leq& 
\frac{\theta^{k}}{\alpha_{k-1}}\left(d_0 +\frac{\alpha_0 \mathcal{R}^\dagger}{\theta -q}\right)+\frac{1}{\alpha_{k-1}}\frac{C}{1-q}\delta^p
\nonumber\\
&=&\theta \left(\frac{d_0}{\alpha_0}+\frac{\mathcal{R}^\dagger}{\theta -q}\right)+\frac{1}{\alpha_0}\frac{C}{1-q}\frac{\delta^p}{\theta^{k-1}}
\nonumber\\
&\leq&
\theta \left(\frac{d_0}{\alpha_0}+\frac{\mathcal{R}^\dagger}{\theta -q}\right)
\left(1+\frac{C}{1-q} \left(\tilde{\tau}-\frac{C}{1-q}\right)^{-1}\right)=:R.
\label{defR}
\end{eqnarray}

It remains to show finiteness of the stopping index for \eqref{IRGNMIvanov}, as boundedness of the $\mathcal{R}$ values by $R=\rho$ holds by definition.
Applying the minimality argument with $\xdag$ being admissible (cf.\eqref{rhok}) to \eqref{IRGNMIvanov} leads to the special case $p=1$, $\alpha_k=0$ in \eqref{genfor}
\[
(1-c_{tc})\|F(\x1)-\ydel\| \leq 2c_{tc}\|F(\xk)-\ydel\|+(1+c_{tc})\delta.
\]
Our notation becomes 
\begin{eqnarray*}
&&d_k:= (1-c_{tc})\|F(\xk)-\ydel\|^2,\\
&&q:=\frac{2c_{tc}}{1-c_{tc}} \in (0,1),\\ 
&&C:=(1+c_{tc}),
\end{eqnarray*}
which gives
\[
d_{k+1} \leq qd_k+C\delta,
\]
and by induction, one can conclude
\[
d_{k} < q^{k}d_0 +\left(\frac{1}{1-q}\right)C\delta,
\]
where the right hand side is smaller than $\tilde{\tau}\delta$ (with $\tilde{\tau}=(1-c_{tc})\tau$) for all
\[
k \geq (\log 1/q)^{-1}\left(p\log (1/\delta) +\log d_0 -\log \left(\tilde{\tau}-\frac{C}{1-q}\right)\right)=:\bar{k}(\delta),
\]
so that we can again conclude $k_{*}(\delta,\ydel)\leq\bar{k}(\delta)=O(\log (1/\delta))$.
\end{proof}

Now we consider the appearance of discretization errors in the numerical solution of \eqref{IRGNMTikhonov}, \eqref{IRGNMIvanov} arising from restriction of the minimization to finite dimensional subspaces $X^k_h$ and leading to discretized iterates $x_{k,h}^\delta$ and an approximate version $F^k_h$ of the forward operator i.e., we consider
the discretized version of Tikhonov-IRGNM \eqref{IRGNMTikhonov} 
\begin{equation}\label{TIRGNMh}
\xh1 \in{\rm argmin}_{x\in \cD (F)\cap X^k_h} \  \|{F^k_h}'(\xkh)(x-\xkh)+F^k_h(\xkh)-\ydel\|^p+\alpha_k\mathcal{R}(x).
\end{equation}
and of Ivanov-IRGNM \eqref{IRGNMIvanov}
\begin{equation}\label{IIRGNMh}
\xh1 \in{\rm argmin}_{x\in \cD (F)\cap X^k_h} \   \frac{1}{2}\|{F^k_h}'(\xkh)(x-\xkh)+F^k_h(\xkh)-\ydel\|^2 \mbox{ such that } \cR(x) \leq \rho,
\end{equation}
respectively.
Moreover, also in the discrepancy principle, the residual is replaced by its actually computable discretized version
\begin{equation}\label{discrprinc_h}
k_*=k_*(\delta,\ydel)=\min\{k\in\N_0\ : \ \|F^k_h(x_{k,h}^{\delta})-\ydel\|\leq\tau\delta\}\,.
\end{equation}

We define the auxiliary continuous iterates  
\begin{equation}\label{TIRGNM}
\x1 \in{\rm argmin}_{x\in \cD (F)} \  \|F'(\xkh)(x-\xkh)+F(\xkh)-\ydel\|^p+\alpha_k\mathcal{R}(x).
\end{equation}
and 
\begin{equation}\label{IIRGNM}
\x1 \in{\rm argmin}_{x\in \cD (F)} \ \|F'(\xkh)(x-\xkh)+F(\xkh)-\ydel\|
\mbox{ such that } \cR(x) \leq \rho,
\end{equation}
respectively in order to be able to use minimality, i.e., compare with the continuous exact solution $\xdag$.
For an illustration we refer to \cite[Figure 1]{KKV14}. 

First of all, we assess how large the discretization errors can be allowed to still enable convergence. Later on, in Section \ref{sec:errest}, we will describe how to really obtain such estimates a posteriori and to achieve the prescribed accuracy by adaptive discretization.

\begin{corollary}\label{cor:conv}
Let the assumptions of Theorem \ref{th:conv} be satisfied 
and assume that the discretization error estimates
\begin{eqnarray} 
&&\|F(x_{k+1,h}^{\delta})-\ydel\|-\|F(x_{k+1}^{\delta})-\ydel\|\,  \leq \eta_{k+1}
\label{eta}\\
&&\left| \|F_h^k(x_{k,h}^{\delta})-\ydel\|-\|F(x_{k,h}^{\delta})-\ydel\|\right|\,  \leq \xi_k
\label{xi}\\
&&\cR(x_{k,h}^{\delta})-\cR(x_{k}^{\delta})\leq \zeta_k 
\label{zeta}
\end{eqnarray}
(note that no absolute value is needed in \eqref{eta}, \eqref{zeta}; moreover, \eqref{zeta} is only be needed for \eqref{IRGNMTikhonov})
hold with
\begin{equation}\label{etaxizeta}
\eta_k\leq c_\eta\|F(x_{k,h}^{\delta})-\ydel\|+\bar{\tau}\delta, \quad\xi_k\leq \hat{\tau}\delta, \quad\zeta_k\leq \bar{\zeta}.
\end{equation}
for all $k\leq k_*(\delta,\ydel)$ and constants $c_\eta,\bar{\tau}>0$ sufficiently small, $\hat{\tau}\in(0,\tau)$, $\bar{\zeta}>0$.

Then the assertions of Theorem \ref{th:conv} remain valid for $x^\delta_{k_*(\delta,\ydel),h}$ in place of $x^\delta_{k_*(\delta,\ydel)}$ with \eqref{discrprinc_h} in place of \eqref{discrprinc}.
\end{corollary}

\begin{proof}
As before, from the minimality of $\x1$ and \eqref{delta}, \eqref{tangcone} as well as $\xdag \in \cD(F)$, we have
\begin{eqnarray*}
&&\Bigl((1-c_{tc})\|F(\x1)-\ydel\|-c_{tc}\|F(\xkh)-\ydel\|\Bigr)^p +\alpha_k\cR(\x1) \\
&&\leq \Bigl(c_{tc}\|F(\xkh)-\ydel\|+(1+c_{tc})\delta\Bigr)^p + \alpha_k\cR(\xdag),
\end{eqnarray*}
then using \eqref{eta}, \eqref{zeta},
\begin{eqnarray*}
&&\Bigl((1-c_{tc})(\|F(\xh1)-\ydel\|-\eta_{k+1})-c_{tc}\|F(\xkh)-\ydel\|\Bigr)^p +\alpha_k\mathcal{R}(\xh1) \\
&&\leq \Bigl(c_{tc}\|F(\xkh)-\ydel\|+(1+c_{tc})\delta\Bigr)^p + \alpha_k\cR(\xdag)+\alpha_k\zeta_{k+1}.
\end{eqnarray*}

Hence, with the same technique as in the proof of Theorem \ref{th:conv}, using \eqref{abab} with $\epsilon=\frac12$, we have
\begin{eqnarray*}
d_{k+1,h}+\alpha_k \cR_{k+1,h} 
&\leq& \tilde{q} d_{k,h} + \alpha_0\theta^k (\cR^\dagger+\zeta_{k+1})+C\delta^p+D\eta_{k+1}^p\\
&\leq& q d_{k,h} + \alpha_0\theta^k (\cR^\dagger+\zeta_{k+1})+(C+D\bar{\tau}^p)\delta^p\,,
\end{eqnarray*}
using \eqref{etaxizeta}, where
\begin{eqnarray*}
&&d_{k,h}:= 2^{1-p}(1-c_{tc})^p\|F(\xkh)-\ydel\|^p,\\
&&\tilde{q}:=\frac{(1+\gamma)^{p-1}+(1+\tilde{\gamma})^{p-1}}{2}\left(\frac{2c_{tc}}{1-c_{tc}}\right)^{p}\,, \quad q=\tilde{q}+Dc_\eta\  
\in (0,1),\\
&&\cR_{k,h}:=\mathcal{R}(\xkh),  \quad \cR^\dagger:=\mathcal{R}(\xdag),\\
&&C:=\left(\frac{1+\gamma}{\gamma}\right)^{p-1}(1+c_{tc})^p, \quad D:=\left(\frac{1+\tilde{\gamma}}{\tilde{\gamma}}\right)^{p-1}(1-c_{tc})^p,
\end{eqnarray*}
for $\gamma, \tilde{\gamma}, c_\eta \in (0,1)$, which are chosen small enough so that $q<\theta$.
From this, by induction we conclude

\begin{equation}\label{indu2}
d_{k+1,h}+\alpha_k\mathcal{R}_{k+1,h} \leq q^{k+1}d_0 +\left(\frac{1}{1-\frac{q}{\theta}}\right)\alpha_k(\mathcal{R}^\dagger+\bar{\zeta})+\left(\frac{1}{1-q}\right)(C+D\bar{\tau}^p)\delta^p
\end{equation} 
Hence, by \eqref{xi}, \eqref{etaxizeta}, we have the following estimate
\[
\|F^k_h(x_{k,h}^\delta)-\ydel\|\leq \hat{\tau}\delta+\left(\frac{2^{p-1}}{(1-c_{tc})^p}
\left(\theta^{k}\left(d_0+\frac{\alpha_0}{\theta -q}(\cR^\dagger+\bar{\zeta})\right)+\frac{C+D\bar{\tau}^p}{1-q}\delta^p\right)\right)^{1/p}\,,
\]
where the right hand side falls below $\tau\delta$ as soon as
\[
k\geq (\log 1/\theta)^{-1}\left(p\log (1/\delta) +\log \left(d_0+\frac{\alpha_0}{\theta -q}(\mathcal{R}^\dagger+\bar{\zeta})\right)-\log \left(\tilde{\tau}-\frac{C+D\bar{\tau}^p}{1-q}\right)\right)=:\bar{k}(\delta),
\]
for $\tilde{\tau}= 2^{1-p}(1-c_{tc})^p(\tau-\hat{\tau})^p$. 
Note that $\tilde{\tau} > \frac{C+D\bar{\tau}^p}{1-q}$, provided $\tau$ is chosen sufficiently large, which we assume to be done. 
That is, we have shown that the discrepancy stopping criterion from \eqref{discrprinc} (with $F$ replaced by $F^k_h$) will be satisfied after finitely many, namely $O(\log(1/\delta))$, steps.

On the other hand, the continuous discrepancy at the iterate defined by the discretized discrepancy principle \eqref{discrprinc_h} by \eqref{xi}, \eqref{etaxizeta} satisfies 
\[
\|F(x_{k,h}^\delta)-\ydel\|\leq (\tau+\hat{\tau})\delta\,.
\]

To estimate $\cR(x^\delta_{k_*(\delta,\ydel),h})$, note that according to our notation, from \eqref{indu2}, we get, like in \eqref{defR}, that for all $k\in \{1,\ldots,k_*(\delta,\ydel)\}$
\begin{eqnarray*}
\mathcal{R}_{k} &\leq& 
\theta \left(\frac{d_0}{\alpha_0}+\frac{\mathcal{R}^\dagger+\bar{\zeta}}{\theta -q}\right)
\left(1+\frac{C+D\bar{\tau}^p}{1-q} \left(\tilde{\tau}-\frac{C+D\bar{\tau}^p}{1-q}\right)^{-1}\right)=:R.
\end{eqnarray*}
It remains to show finiteness of the stopping index for the discretized Ivanov-IRGNM \eqref{IIRGNMh}.
By minimality of $\x1$ and \eqref{eta}, for this problem we have
\[
(1-c_{tc})\|F(\xh1)-\ydel\| \leq 2c_{tc}\|F(\xkh)-\ydel\|+(1+c_{tc})\delta +(1-c_{tc})\eta_{k+1}.
\]
which with 
\begin{eqnarray*}
&&d_{k,h}:= (1-c_{tc})\|F(\xkh)-\ydel\|^2,\\
&&\tilde{q}:=\frac{2c_{tc}}{1-c_{tc}}\,, \quad q=\tilde{q}+Dc_\eta\ \in (0,1),\\ 
&&C:=(1+c_{tc}), \quad D:=(1-c_{tc}),
\end{eqnarray*}
by induction, \eqref{xi} and \eqref{etaxizeta} gives
\[
\|F^k_h(x_{k,h}^\delta)-\ydel\|\leq \frac{1}{1-c_{tc}} d_{k,h} +\xi_{k} \leq \frac{1}{1-c_{tc}}\left(q^{k}d_0 + \frac{C+D\bar{\tau}}{1-q}\delta\right)+\hat{\tau}\delta,
\]
where the right hand side is smaller than $\tau\delta$ for all
\[
k \geq (\log 1/q)^{-1}\left(p\log (1/\delta) +\log d_0 -\log \left(\tilde{\tau}-\frac{C+D\bar{\tau}}{1-q}\right)\right)=:\bar{k}(\delta),
\]
with $\tilde{\tau}=(1-c_{tc})(\tau-\hat{\tau})$, so that we can again conclude $k_{*}(\delta,\ydel)\leq\bar{k}(\delta)=O(\log(1/\delta))$.
\end{proof}

\section{Error estimators for adaptive discretization}\label{sec:errest}

The error estimators $\eta_k$, $\xi_k$ and $\zeta_k$ can be quantified, e.g., by means of a goal oriented dual weighted residual (DWR) approach \cite{BeVex}, applied to the minimization problems 
\begin{eqnarray}\label{PDEminTikhonov}
\hspace*{-1cm}(\xh1,v_{k,h}^\delta,u^\delta_{k+1},u_{k,h}^\delta)&&\in{\rm argmin}_{(x,v,u,\tilde{u})\in\mathcal{D}{(F)}\times V^3} \|C'(\tilde{u})v+C(\tilde{u})-\ydel\|^p+\alpha_k\mathcal{R}(x)\\ \nonumber
\mbox{ s.t. } \forall w\in W: &&  
\langle A'_x(x_{k,h}^\delta,\tilde{u})(x-x_{k,h}^\delta)+A'_u(x_{k,h}^\delta,\tilde{u})v,w\rangle_{W^*,W}=0,\\\nonumber
&& \langle A(x_{k,h}^\delta,\tilde{u}),w\rangle_{W^*,W}=0, \quad \langle A(x,u),w\rangle_{W^*,W}=0,
\end{eqnarray}
(note that the last constraint is added in order to enable computation of $I_2^k$ below) and
\begin{eqnarray}\label{PDEminIvanov}
(\xh1,v_{k,h}^\delta,u^\delta_{k+1},u_{k,h}^\delta)&&\in{\rm argmin}_{(x,v,u,\tilde{u})\in\mathcal{D}{(F)}\times V^3} \frac{1}{2}\|C'(\tilde{u})v+C(\tilde{u})-\ydel\|^2\\\nonumber
\mbox{ s.t. } &&\mathcal{R}(x)\leq\rho_k,\\\nonumber
 \mbox{ and } \forall w\in W: &&
\langle A'_x(x_{k,h}^\delta,\tilde{u})(x-x_{k,h}^\delta)+A'_u(x_{k,h}^\delta,\tilde{u})v,w\rangle_{W^*,W}=0,\\\nonumber
&& \langle A(x_{k,h}^\delta,\tilde{u}),w\rangle_{W^*,W}=0, \quad \langle A(x,u),w\rangle_{W^*,W}=0,
\end{eqnarray}
which are equivalent to \eqref{IRGNMTikhonov}, \eqref{IRGNMIvanov}, respectively,
with 
\[
I_1^k(x,v,u,\tilde{u})=\|C(\tilde{u})-\ydel\|\,, \quad 
I_2^k(x,v,u,\tilde{u})=\|C(u)-\ydel\|\,,\quad
I_3^k(x,v,u,\tilde{u})= \mathcal{R}(x)
\]
as quantities of interest (where $I_3^k$ is only needed for \eqref{IRGNMTikhonov}).
We assume that $C,\mathcal{R}$ and the norms can be evaluated without discretization error, so the discretized versions of $I_i^k$ only arise due to discreteness of the arguments.
Indeed, it is easy to see that the left hand sides of \eqref{eta} and \eqref{xi} can be bounded (at least approximately) by combinations of $I_1^k$ and $I_2^k$, using the triangle inequality: 
\begin{eqnarray} 
&&\|F(x_{k+1,h}^{\delta})-\ydel\|-\|F(x_{k+1}^{\delta})-\ydel\|
\nonumber\\
&&=I_1^{k+1}(x_{k+2}^\delta, v_{k+1}^\delta, u_{k+2}^\delta, \tilde{u}_{k+1}^\delta)
-I_1^{k+1}(x_{k+2,h}^\delta, v_{k+1,h}^\delta, u_{k+2,h}^\delta, \tilde{u}_{k+1,h}^\delta)
\nonumber\\
&&\quad-(I_2^{k}(x_{k+1}^\delta, v_{k}^\delta, u_{k+1}^\delta, \tilde{u}_{k}^\delta)
-I_2^{k}(x_{k+1,h}^\delta, v_{k,h}^\delta, u_{k+1,h}^\delta, \tilde{u}_{k,h}^\delta))
+R_\eta^{k+1};
\label{etaI}\\
&&\|F_h^k(x_{k,h}^{\delta})-\ydel\|-\|F(x_{k,h}^{\delta})-\ydel\|
\nonumber\\
&&=I_1^{k}(x_{k+1,h}^\delta, v_{k,h}^\delta, u_{k+1,h}^\delta, \tilde{u}_{k,h}^\delta)
-I_1^{k}(x_{k+1}^\delta, v_{k}^\delta, u_{k+1}^\delta, \tilde{u}_{k}^\delta)
\label{xiI}\,,
\end{eqnarray}
where we will neglect $R_\eta^{k+1}=\|F_h^{k+1}(x_{k+1,h}^{\delta})-\ydel\|-\|F_h^k(x_{k+1,h}^{\delta})-\ydel\|$.

It is important to note that $I_{1,h}^{k+1}$ is not equal to $I_{2,h}^k$, see \cite{KKV14}.
%The estimators obtained by this procedure can be used to trigger local mesh refinement until the requirements \eqref{etaxizeta} are met.   

The computation of the a posteriori error estimators $\eta_k, \xi_k, \zeta_k$ is done as in \cite{KKV14}. These error estimators can be used within the following adaptive algorithm for error control and mesh refinement: We start on a coarse mesh, solve the discretized optimization problem and evaluate the error estimator. Thereafter, we refine the current mesh using local information obtained from the error estimator, reducing the error with respect to the quantity of interest. This procedure is iterated until the value of the error estimator is below the given tolerance \eqref{etaxizeta}, cf. \cite{BeVex}.

In this case, all the variables $x,v,u,\tilde{u}$ are subject to a new discretization. For better readability we will partially omit the iteration index $k$ and the discretization index $h$. The previous iterate $\xk$ is fixed and not subject to a new discretization. 

Consider now the cost functional for \eqref{PDEminTikhonov}
\[
J(x,v,\tilde{u}) = \|C'(\tilde{u})v+C(\tilde{u})-\ydel\|^p+\alpha_k\mathcal{R}(x)
\]
and define the Langrangian functional
\begin{eqnarray}\label{LagTi}\nonumber
L(x,v,u,\tilde{u},\lambda,\tilde{\mu},\mu):=J(x,v,\tilde{u})+\langle A'_x(x_k^\delta,\tilde{u})(x-x_k^\delta)+A'_u(x_k^\delta,\tilde{u})v,\lambda\rangle_{W^*,W}\\ +\langle A(x_k^\delta,\tilde{u}),\tilde{\mu}\rangle_{W^*,W}+\langle A(x,u),\mu\rangle_{W^*,W}\,,
\end{eqnarray}
assuming for simplicity that $\mathcal{D}(F)=X$.
The first-order necessary optimality conditions for \eqref{PDEminTikhonov} are given by  stationarity for the Lagrangian $L$. Setting $z=(x,v,u,\tilde{u},\lambda,\tilde{\mu},\mu)$, it reads 
\[
L'(z)(dz)=0, \forall dz \in Z = X\times V\times V\times V\times  W\times W\times W
\]
and for the discretized problem, 
\[
L'(z_h)(dz_h)=0, \forall dz_h \in Z_h = X_h\times V_h\times V_h\times V_h \times W_h\ \times W_h\times W_h\,.
\]
To derive a posteriori error estimators for the error with respect to the quantities of interest ($I_1, I_2, I_3$), we introduce auxiliary functionals $M_i$:
\[ 
M_i(z,\bar{z})= I_i(z)+L'(z)\bar{z}, \quad z,\bar{z} \in Z, \quad i=1,2,3, 
\]
Let $\tilde{z}=(z,\bar{z}) \in \tilde{Z} = Z \times Z$ and $\tilde{z}_h=(z_h,\bar{z}_h) \in \tilde{Z}_h = Z_h \times Z_h$ be continuous and discrete stationary points of $M_i$ satisfying
\[
M'(\tilde{z})(d\tilde{z})=0, \forall d\tilde{z} \in Z \qquad 
M'(\tilde{z}_h)(d\tilde{z}_h)=0, \forall d\tilde{z}_h \in Z_h\,,
\]
respectively.
Then, $z, z_h$ are continuous and discrete stationary points of $L$ and there holds $I_i(z)=M_i(\tilde{z}), i = 1,2,3$. Thus the $z$ part, as computed already during the numerical solution of the minimization problem \eqref{PDEminTikhonov} (or \eqref{PDEminIvanov}) remains fixed for all $i\in\{1,2,3,\}$. Moreover, after computing the discrete stationary point $z_h$ for $L$ (e.g., by applying Newton's method), it requires only one more Newton step to compute the $\bar{z}$ coordinate of the stationary point for $M$ from 
\[
L''(z_h)(\bar{z}_{i,h},d\bar{z})=-I_i'(z_h)d\bar{z} , \forall d\tilde{z}_h \in Z_h.
\]
According to \cite{BeVex}, there holds
\[
I_i(x,v,\tilde{u})-I_i(x_h,v_h,\tilde{u}_h)=\frac{1}{2}M'(\tilde{z}_h)(\tilde{z}-\hat{z}_h)+R, \quad \forall \hat{z}_h \in Z_h \quad i=1,2,3,
\]
with a remainder term $R$ of order $O(\|\tilde{z}-\tilde{z}_h\|^3)$ that is therefore neglected.
Thus we use
\[
I_i^k(z)-I_i^k(z_h) \approx \frac{1}{2}M_i'(z_h,\bar{z}_{i,h})(\pi_h\tilde{z}_{i,h}-\tilde{z}_{1,h})=\varepsilon^k_i,
\]
where $\pi_h$ is an operator to approximate the interpolation error as in \cite{KKV14}, typically defined by local averaging, to define the estimators $\eta_k$, $\xi_k$, $\zeta_k$ according to the rule
\begin{equation}\label{etaepszeta}
\eta_{k+1}= \varepsilon^{k+1}_1 +\varepsilon^{k}_2\,, \quad \xi_k=\varepsilon^{k}_1\,, \quad \zeta_k=\varepsilon^{k}_3;
\end{equation}
cf. \eqref{etaI}, \eqref{xiI}. The estimators obtained by this procedure can be used to trigger local mesh refinement until the requirements \eqref{etaxizeta} are met.   

Explictly, for $p=2$ (for simplicity) such a stationary point $z=(x,v,u,\tilde{u},\lambda,\tilde{\mu})$ can be computed by solving the following system of equations (analogously for the discrete stationary point of $L$)
\begin{eqnarray}\label{firsteq}
-(A^{'}_x(x,u)^*\mu+A^{'}_{x} (x_k^\delta,\tilde{u})^{*}\lambda) \in \alpha_k \partial \mathcal{R}(x); &&\\
 2\langle C^{'}(\tilde{u})(dv),C^{'}(\tilde{u})v+C(\tilde{u})-y^\delta\rangle +\langle A^{'}_u(x_k^\delta,\tilde{u})(dv),\lambda\rangle=0, &\forall dv \in V;&\\\label{thirdeq}
 \langle A^{'}_u(x,u)(du),\mu\rangle=0, &\forall du \in V;&\\
\langle A^{''}_{xu} (x_k^\delta,\tilde{u})(x-x_k^\delta,d\tilde{u})+A^{''}_{uu} (x_k^\delta,\tilde{u})(v,d\tilde{u}),\lambda\rangle +\langle A^{'}_u(x^\delta_k,\tilde{u})(d\tilde{u}),\tilde{\mu}\rangle \quad \quad \; \nonumber\\
+2\langle C^{''}(\tilde{u})(d\tilde{u},v)+C^{'}(\tilde{u})(d\tilde{u}),C^{'}(\tilde{u})v+C(\tilde{u})-y^\delta\rangle =0, &\forall d\tilde{u} \in V;&\\
\langle A^{'}_x (x_k^\delta,\tilde{u})(x-x_k^\delta)+A^{'}_u (x_k^\delta,\tilde{u})v,d\lambda\rangle = 0, &\forall d\lambda\in W;&\\
\langle A(x_k^\delta,\tilde{u}),d\tilde{\mu}\rangle=0, &\forall d\tilde{\mu} \in W;&\\\label{lasteq}
\langle A(x,u),d\mu\rangle=0, &\forall d\mu \in W.&
\end{eqnarray}
Note that \eqref{lasteq} is decoupled from the other equations and that if $A^{'}_u(x,u)^*$ is injective, equation \eqref{thirdeq} implies $\mu=0$.

Summarizing, since we have a convex minimization problem, after solving a nonlinear system of seven equations to find the minimizer, we need only one more Newton step to compute the error estimators to check whether we need a refinement on the mesh or not.

Regarding the problem \eqref{PDEminIvanov}, we have the following Lagrangian functional
\begin{eqnarray}\label{LagIv}\nonumber
L(x,v,u,\tilde{u},\lambda,\tilde{\mu},\mu):=J(x,v,\tilde{u})+\langle A'_x(x_k^\delta,\tilde{u})(x-x_k^\delta)+A'_u(x_k^\delta,\tilde{u})v,\lambda\rangle_{W^*,W}\\
+\langle A(x^\delta_k,\tilde{u}),\tilde{\mu}\rangle_{W^*,W}+\langle A(x,u),\mu\rangle_{W^*,W},
\end{eqnarray}
where we rewrite the cost functional $J(x,v,\tilde{u})$ for \eqref{PDEminIvanov} as 
\[
J(x,v,\tilde{u}) = \frac{1}{2}\|C'(\tilde{u})v+C(\tilde{u})-\ydel\|^{2} +I_{(-\infty,0]}(\mathcal{R}(x)-\rho);
\]
and the indicator functional $I_{(-\infty,0]}(\mathcal{R}(x)-\rho)$ takes the role of a regularization functional. The following optimality system is the same as above, just with \eqref{firsteq} replaced by
\begin{equation}\label{firsteq2}
-(A^{'}_x(x,u)^*\mu+A^{'}_{x} (x_k^\delta,\tilde{u})^{*}\lambda) \in \partial I_{(-\infty,0]}(\mathcal{R}(x)-\rho).
\end{equation}

Note that the bound on $I_2$ only appears -- via  \eqref{etaepszeta} -- in connection to the assumption $\eta_k \leq \bar{\tau}\delta$, for $k \leq k_*(\delta,\ydel)$ in (\ref{etaxizeta}). This may be satisfied in practice without refining explicitly with respect to $\eta_k$, but simply by refining with respect to the other error estimators $\xi_k$ (and $\zeta_k$ in the Tikhonov case). 
The fact that $I^k_{1,h}$ and $I_{2,h}^{k-1}$ only differ in the discretization level, motivates the assumption that for small $h$, we have $I^k_{1,h} \approx I_{2,h}^{k-1}$ and $\eta_{k-1} \approx \xi_k$. Thefore, the algorithm used in actual computations will be built neglecting $I_2$ and hence skipping the constraint $\langle A(x,u),w\rangle_{W^*,W}=0, \;\forall w \in W$ in (\ref{PDEminTikhonov}) and (\ref{PDEminIvanov}), which implies modifications on the Lagrangians (\ref{LagTi}) and (\ref{LagIv}).
Therefore, the corresponding optimality systems for $p=2$ in the Tikhonov case is given by 
\begin{eqnarray}\label{x}
-A^{'}_{x} (x_k^\delta,\tilde{u})^{*}\lambda \in \alpha_k \partial \mathcal{R}(x); &&\\\label{v}
 2\langle C^{'}(\tilde{u})(dv),C^{'}(\tilde{u})v+C(\tilde{u})-y^\delta\rangle +\langle A^{'}_u(x_k^\delta,\tilde{u})(dv),\lambda\rangle=0, &\forall dv \in V;&\\\label{mu}
\langle A^{''}_{xu} (x_k^\delta,\tilde{u})(x-x_k^\delta,d\tilde{u})+A^{''}_{uu} (x_k^\delta,\tilde{u})(v,d\tilde{u}),\lambda\rangle +\langle A^{'}_u(x^\delta_k,\tilde{u})(d\tilde{u}),\tilde{\mu}\rangle \quad \quad \; \nonumber\\
+2\langle C^{''}(\tilde{u})(d\tilde{u},v)+C^{'}(\tilde{u})(d\tilde{u}),C^{'}(\tilde{u})v+C(\tilde{u})-y^\delta\rangle =0, &\forall d\tilde{u} \in V;&\\\label{lambda}
\langle A^{'}_x (x_k^\delta,\tilde{u})(x-x_k^\delta)+A^{'}_u (x_k^\delta,\tilde{u})v,d\lambda\rangle = 0, &\forall d\lambda\in W;&\\\label{u}
\langle A(x_k^\delta,\tilde{u}),d\tilde{\mu}\rangle=0, &\forall d\tilde{\mu} \in W.&
\end{eqnarray}
Note that equation \eqref{u} is decoupled from the others. Therefore, the strategy is to solve \eqref{u} first, then solve the linear system \eqref{x},\eqref{v},\eqref{lambda} for $(x,v,\lambda)$, and finally compute $\tilde{\mu}$ via the  linear equation \eqref{mu}. 
Here, the system \eqref{x},\eqref{v},\eqref{lambda} can be interpreted as the optimality conditions for the following problem
\begin{eqnarray*}
(\xh1,v_{k,h}^\delta)&&\in{\rm argmin}_{(x,v)\in\mathcal{D}{(F)}\times V} \|C'(\tilde{u})v+C(\tilde{u})-\ydel\|^2+\alpha_k\mathcal{R}(x)\\
\mbox{ s.t. } \forall w\in W: &&  
\langle A'_x(x_{k,h}^\delta,\tilde{u})(x-x_{k,h}^\delta)+A'_u(x_{k,h}^\delta,\tilde{u})v,w\rangle_{W^*,W}=0.
\end{eqnarray*}

For the Ivanov case, we have to solve \eqref{v}-\eqref{u} with 
\begin{equation}\label{x2}
-A^{'}_{x} (x_k^\delta,\tilde{u})^{*}\lambda \in \partial I_{(-\infty,0]}(\mathcal{R}(x)-\rho)
\end{equation} 
in place of \eqref{x}, hence again \eqref{u} is decoupled from the other equations, \eqref{mu} is linear with respect to $\tilde{\mu}$, once $(x,v,\lambda)$ has been computed, and the remaining system for $(x,v,\lambda)$ can be interpreted as the optimality conditions for the following problem
\begin{eqnarray*}
(\xh1,v_{k,h}^\delta)&&\in{\rm argmin}_{(x,v)\in\mathcal{D}{(F)}\times V} \frac{1}{2}\|C'(\tilde{u})v+C(\tilde{u})-\ydel\|^2\\
\mbox{ s.t. } &&\mathcal{R}(x)\leq\rho_k,\\
 \mbox{ and } \forall w\in W: &&
\langle A'_x(x_{k,h}^\delta,\tilde{u})(x-x_{k,h}^\delta)+A'_u(x_{k,h}^\delta,\tilde{u})v,w\rangle_{W^*,W}=0.
\end{eqnarray*}

\begin{remark}
Since DWR estimators are based on residuals which are computed in the optimization process, the additional costs for estimation are very low, which makes this approach attractive for our purposes. However, although these error estimators are known to work efficiently in practice (see \cite{BeVex}), they are not reliable, i.e., the conditions $I^k_i(z)-I^k_i(z_h) \leq \epsilon_i^k$, $i=1,2,3$ can not be guaranteed in a strict sense in the computations, since we neglect the remainder term $R$ and use an approximation for $\tilde{z}-\hat{z}_h$. As our analysis in Theorem \ref{th:conv} is kept rather general, it is not restricted to DWR estimators and would also work with different (e.g., reliable) error estimators. 
\end{remark}

\section{Model Examples}\label{sec:modex}

We present a model example to illustrate the abstract setting from the previous section. Consider the following inverse source problem for a semilinear elliptic PDE, where the model and observation equations are given by
\begin{eqnarray}
-\Delta u +\kappa u^3 &=\chi_{\omega_c}\src,& \mbox{in } \Omega \subset \mathbb{R}^d, \label{modelPDE}\\
u&=0,& \mbox{on } \partial \Omega, \label{modelbc}\\
C(u) &= u\mid_{\omega_o} , & \|y-y^\delta\|_{L^2(\omega_o)} \leq \delta.\label{modelobs}
\end{eqnarray}

We first of all consider Tikhonov regularization and therefore use the space of Radon measures $\mathcal{M}(\omega_c)$ as a preimage space $X$. Thus we define the operators $A: \mathcal{M}(\omega_c)\times W^{1,q^{'}}_0 (\Omega) \longrightarrow W^{-1,q} (\Omega)$, $A(\src,u) = - \Delta u + \kappa u^3 -\src$, $\kappa \in \mathbb{R}$ and the injection $C:W^{1,q^{'}}_0(\Omega)\longrightarrow L^2(\omega_o)$, $q > d$, where $\Omega$ is a bounded domain in $\mathbb{R}^d$ with $d=2$ or $3$, with Lipschitz boundary $\partial \Omega$ and $\omega_c, \omega_o \subset \Omega$ are the control domain and the observation domain, respectively.

A monotonicity argument yields well posedness of the above semilinear boundary value problem, i.e., well-definedness of $u\in W^{1,q^{'}}_0(\Omega)$ as a solution to the elliptic boundary value problem \eqref{modelPDE}, \eqref{modelbc}, as long as we can guarantee that $u^3\in W^{-1,q} (\Omega)$ for any $u\in W^{1,q^{'}}_0(\Omega)$, i.e., the embeddings $W^{1,q^{'}}_0(\Omega)\to L^{3r}(\Omega)$ and $L^r(\Omega)\to W^{-1,q} (\Omega)$ are continuous for some $r\in[1,\infty]$, which (by duality) is the case iff $W^{1,q^{'}}_0(\Omega)$ embeds continuously both into $L^{3r}(\Omega)$ and $L^{r'}(\Omega)$. By Sobolev's Embedding Theorem, this boils down to  the inequalities
\[ 
1-\frac{d}{q'}\geq -\frac{d}{3r}\mbox{ and } 1-\frac{d}{q'}\geq -\frac{d}{r'}\,,
\]
which by elementary computations turns out to be equivalent to 
\begin{equation}\label{r}
\frac{dq}{q+d}\leq r \leq \frac{dq}{3(dq-q-d)}\,,
\end{equation}
where the left hand side is larger than one and the denominator on the right hand side is positive due to the fact that for $d\geq2$ we have $q>d\geq d'=\frac{d}{d-1}$.
Taking the extremal bounds for $q>d$ -- note that the lower bound is increasing and the upper bound is decreasing with $q$ -- in \eqref{r1} we get
\begin{equation}\label{r1}
\frac{d}{2}< r < \frac{d}{3(d-2)}\,.
\end{equation}
Thus, as a by-product, we get that for any $t\in[1,\bar{t})$ there exists $q>d$ such that $W^{1,q^{'}}_0(\Omega)$ continuously embeds into $L^t$, with 
\begin{equation}\label{tbar}
\bar{t}=\infty\mbox{ in case }d=2\mbox{ and }\bar{t}=3\mbox{ in case }d=3\,.
\end{equation}

For the regularization functional $\mathcal{R}(\src)=\|\src\|_{\mathcal{M}(\omega_c)}$, the IRGNM-Tikhonov minimization step is given by (ignoring $h$ in the notation)
\begin{eqnarray*}
(\src_{k+1}^\delta,v_{k}^\delta,u_{k}^\delta)&&\in{\rm argmin}_{(\src,v,\tilde{u})\in\mathcal{M}(\omega_c)\times (W^{1,q^{'}}_0(\Omega))^2} \|C(v+\tilde{u})-\ydel\|^2_{L^2(\omega_o)} +\alpha_k\|\src\|_{\mathcal{M}(\omega_c)}\\
\mbox{ s.t. } \forall w\in W^{1,q^{'}}_{0} (\Omega): && \int_{\Omega}(\nabla v \nabla w+3\kappa \tilde{u}^2 vw)d\Omega = \int_{\Omega} wd(\src-\src_k^\delta),\\
&& \int_{\Omega} (\nabla \tilde{u}\nabla w+\kappa \tilde{u}^3w)d\Omega= \int_{\Omega} wd\src_k^\delta .
\end{eqnarray*}
Therefore, to compute this Gauss-Newton step, one first needs to solve the equation
\begin{equation}\label{nonlinear}
-\Delta \tilde{u}+\kappa\tilde{u}^3=\src_k^\delta,
\end{equation}
then solve the following optimality system with respect to $(\src,v,\lambda)$ (written in a strong formulation)
\begin{eqnarray*}
\|\lambda\|_{C_b(\omega_c)} \leq \alpha_k  \mbox{ and } \int_\Omega (\src^{*} -\lambda )d\src & \leq& 0, \forall \src^{*}\in B_{\alpha_k}^{C_b (\omega_c)}\\
-\Delta \lambda +3\kappa \tilde{u}^2\lambda +2v +2\tilde{u}&=&2\ydel \\
-\Delta v +3\kappa \tilde{u}^2v -\src&=&-\src^\delta_k,
\end{eqnarray*} 
which can be interpreted as the optimality system for the minimization problem
\begin{eqnarray}\label{subtikhonov}
(\src_{k+1}^\delta,v_{k}^\delta)&&\in{\rm argmin}_{(\src,v)\in \mathcal{M}(\omega_c)\times W^{1,q^{'}}_0(\Omega)} \|\tilde{u}+v-\ydel\|^2_{L^2(\omega_o)}+\alpha_k\|\src\|_{\mathcal{M}(\omega_c)}\\\nonumber
\mbox{ s.t. } && -\Delta v +3\kappa \tilde{u}^2 v = \src-\src^\delta_k,
\end{eqnarray}
and finally, compute $\tilde{\mu}$ by solving
\begin{equation}\label{muequation}
-\Delta \tilde{\mu} +3\kappa \tilde{u}^2\tilde{\mu}=-6\kappa\tilde{u}v\lambda -2(v+\tilde{u}-\ydel).
\end{equation}
For carrying out the IRGNM iteration, $\tilde{\mu}$ is not required, but we need it for evaluating the error estimators.

For the Ivanov case, we consider the same model and observation equations \eqref{modelPDE}, \eqref{modelbc}, \eqref{modelobs} but now we intend to regularize by imposing $L^\infty$ bounds and thus use the slightly different function space setting, $A: L^\infty(\omega_c)\times H^{1}_0 (\Omega) \longrightarrow H^{-1} (\Omega)$, $A(\src,u) = - \Delta u + \kappa u^3 -\src$, $\kappa \in \mathbb{R}$ and the injection $C: H^1_0(\Omega)\longrightarrow L^2(\omega_o)$.

The IRGNM-Ivanov minimization step with the regularization functional $\mathcal{R}(\src)=\|\src\|_{L^\infty(\omega_c)}$ is given by (ignoring the $h$ in the notation)
\begin{eqnarray*}
(\src_{k+1}^\delta,v_{k}^\delta,u_{k}^\delta)&&\in{\rm argmin}_{(\src,v,\tilde{u})\in L^\infty(\omega_c)\times (H^1_0(\Omega))^2} \|C(v+\tilde{u})-\ydel\|^2_{L^2(\omega_o)} \\
\mbox{ s.t. } && \|\src\|_{L^\infty(\omega_c)} \leq \rho \\
\mbox{ and } \forall w\in H^1_0 (\Omega): && \int_{\Omega}(\nabla v \nabla w+3\kappa \tilde{u}^2 vw)d\Omega = \int_{\Omega} w(\src-\src_k^\delta)d\Omega,\\
&& \int_{\Omega} (\nabla \tilde{u}\nabla w+\kappa \tilde{u}^3wd\Omega= \int_{\Omega} w\src_k^\delta d\Omega .
\end{eqnarray*}
For the Gauss-Newton step, one needs to first solve the equation \eqref{nonlinear} and then, solve the following optimality system with respect to $(\src,v,\lambda)$ (written in a strong formulation)
\begin{eqnarray*}
\|\src\|_{L^{\infty}(\omega_c)} \leq \rho  \mbox{ and } \int_\Omega (\src^{*} -\src)\lambda d\Omega & \leq& 0, \forall \src^{*}\in B_{\rho}^{L^{\infty}(\omega_c)}\\
-\Delta \lambda +3\kappa \tilde{u}^2\lambda +2v +2\tilde{u}&=&2\ydel \\
-\Delta v +3\kappa \tilde{u}^2v -\src&=&-\src^\delta_k,
\end{eqnarray*} 
which can be interpreted as the optimality system for the minimization problem
\begin{eqnarray}\label{supivanov}
(\src_{k+1}^\delta,v_{k}^\delta)&&\in{\rm argmin}_{(\src,v)\in L^{\infty}(\omega_c)\times H^1_0(\Omega)} \frac{1}{2}\|\tilde{u}+v-\ydel\|^2_{L^2(\omega_o)}\\\nonumber
\mbox{ s.t. } && \|\src\|_{L^\infty(\omega_c)} \leq \rho \\\nonumber
 && -\Delta v +3\kappa \tilde{u}^2 v = \src-\src^\delta_k\,.
\end{eqnarray}
Finally, $\tilde{\mu}$ is computed from \eqref{muequation}.

For numerically efficient methods to solve the minimization problems \eqref{subtikhonov} and \eqref{supivanov} we refer to e.g. \cite{ClaKun,ClaKun2,ClaKunCasas} and the references therein.

We finally check the tangential cone condition in case $\omega_0=\Omega$ in both settings
\[
X=\mathcal{M}(\omega_c)\,, \quad V=W^{1,q'}_0(\Omega)\,, \quad W=W^{1,q}_0(\Omega)
\]
(where we will have to restrict ourselves to $d=2$) and 
\[
X=L^\infty(\omega_c)\,, \quad V=W=H_0^1(\Omega)\,.
\]
For this purpose, we use the fact that with the notation 
$F(\tilde{\src})=\tilde{u}\vert_{\omega_o}$, $F(\src)=u\vert_{\omega_o}$, 
$F(\tilde{\src})-F(\src)=v\vert_{\omega_o}$ and 
$F(\tilde{\src})-F(\src)-F'(\src)(\tilde{\src}-\src)=w\vert_{\omega_o}$, 
the functions $v,w\in W^{1,q^{'}}_0(\Omega)$ satisfy the homogeneous Dirichlet boundary value problems for the equations
\[
-\Delta v+\kappa(\tilde{u}^2+\tilde{u}u+u^2)\, v =\tilde{\src}-\src
\]
\[
-\Delta w+\kappa u^2 w = -\kappa(\tilde{u}+2u)\, v^2\,.
\]
Using an Aubin-Nitsche type duality trick, we can estimate the $L^2$ norm of $w$ via the adjoint state $p\in W_0^{1,n}(\Omega)$, which solves 
\[
-\Delta p+\kappa u^2 p = w\,,
\]
with homogeneous Dirichlet boundary conditions, so that by H\"older's inequality
\begin{eqnarray*}
&&\|w\|_{L^2(\Omega)}^2=\langle w,(-\Delta +\kappa u^2\mbox{id}) p\rangle
=\langle (-\Delta +\kappa u^2\mbox{id})w, p\rangle\\
&&= -\kappa\langle(\tilde{u}+2u)\, v^2, p\rangle
\leq \kappa \|v\|_{L^2(\Omega)} \|\tilde{u}+2u\|_{L^m(\Omega)} \|v\|_{L^m(\Omega)} 
\|p\|_{L^{\frac{2m}{m-4}}(\Omega)}\\
&&\leq \tilde{\tilde{C}} \kappa \|v\|_{L^2(\Omega)} \|\tilde{u}+2u\|_{L^m(\Omega)} \|v\|_{L^m(\Omega)} 
\|w\|_{L^2(\Omega)}\,,
\end{eqnarray*}
where we aim at choosing $m\in[4,\infty]$, $n\in[1,\infty]$ such that indeed 
\[
\|p\|_{W_0^{1,n}(\Omega)}\leq C \|w\|_{W^{-1,n}(\Omega)}\leq \tilde{C} \|w\|_{L^2(\Omega)}
\]
and the embeddings $V\to L^m(\Omega)$, $W^{1,n}(\Omega)\to L^{\frac{2m}{m-4}}(\Omega)$, $L^2(\Omega)\to W^{-1,n}(\Omega)$ are continuous. If we succeed in doing so, we can bound 
$\tilde{\tilde{C}} \kappa \|\tilde{u}+2u\|_{L^m(\Omega)} \|v\|_{L^m(\Omega)}$ by some constant $c_{tc}$, which will be small provided $\|\tilde{\src}-\src\|_X$ and hence $\|v\|_{L^m(\Omega)}$ is small. 
Thus, the numbers $n,m$ are limited by the requirements 
\begin{equation}\label{embeddings1}
V\subseteq L^m(\Omega)
\mbox{ and }W^{1,n}(\Omega)\subseteq L^{\frac{2m}{m-4}}(\Omega)
\mbox{ and } m\geq 4\,,
\end{equation}
$L^2(\Omega)\subseteq W^{-1,n}(\Omega)$, i.e., by duality, 
\begin{equation}\label{embeddings2}
W_0^{1,n'}(\Omega)\subseteq L^2(\Omega)\,,
\end{equation} 
and the fact that $\kappa u^2 p\in L^o(\Omega)$ should be contained in $W^{-1,n'}(\Omega)$ for $u\in V\subseteq L^t(\Omega)$, and $p\in W^{1,n}(\Omega)$, which 
via H\"older's inequality in 
\[
\left(\int_\Omega (u^2p)^o\, d\Omega\right)^{1/o} \leq \|u\|_{L^t(\Omega)}^2 \|p\|_{L^\frac{ot}{t-2o}(\Omega)}
\]
and duality leads to the requirements 
\begin{equation}\label{embeddings3}
W_0^{1,n}(\Omega)\subseteq L^{o'}(\Omega)\mbox{ and }
V\subseteq L^t(\Omega)\mbox{ and }
W_0^{1,n}(\Omega)\subseteq L^\frac{ot}{t-2o}(\Omega)\mbox{ and }
o\leq \frac{t}{2}
\end{equation} 
In case $V=W_0^{1,q'}(\Omega)$ with $q>d$ and $d=3$, \eqref{embeddings1} will not work out, since according to \eqref{tbar}, $m$ cannot be chosen larger or equal to four.\\
In case $V=W_0^{1,q'}(\Omega)$ with $q>d$ and $d=2$, we can choose, e.g., $t=m=n=6$, $o=2$ to satisfy \eqref{embeddings1}, \eqref{embeddings2}, \eqref{embeddings3} as well as $t,m<\bar{t}$ as in \eqref{tbar}.\\ 
The same choice is possible in case $V=H_0^1(\Omega)$ with $d\in\{2,3\}$.

\section{Numerical tests}\label{sec:num}

In this section, we provide some numerical illustration of the IRGNM Ivanov method applied to the example from section \ref{sec:modex}, i.e., each Newton step consists of solving \eqref{nonlinear} and subsequently \eqref{supivanov}. For the numerical solution of \eqref{nonlinear} we apply a damped Newton iteration to the equation $\Phi(\tilde{u})=0$ where 
\[
\Phi: H_0^1(\Omega)\to H^{-1}(\Omega)\,,\quad \Phi(\tilde{u})=-\Delta \tilde{u}+\kappa\tilde{u}^3-\src^\delta_k\,,
\] 
\[
\tilde{u}^{l+1}=\tilde{u}^{l}-\Bigl(-\Delta \tilde{u}+3\kappa(\tilde{u}^l)^2\Bigr)^{-1}\Bigl(-\Delta \tilde{u}+\kappa(\tilde{u}^l)^3-\src_k^\delta\Bigr)\,,
\]
which is stopped as soon as $\|\Phi(\tilde{u}^l)\|_{H^{-1}(\Omega)}$ has been reduced by a  factor of $1.e-4$.
The sources $\src$ and states $u$ are discretized by piecewise linear finite elements, hence after elimination of the state via the linear equality constraint, \eqref{supivanov} becomes a box constrained quadratic program for the dicretized version of $\src$, which we solve with the method from \cite{HungerlaenderRendl15}
using the Matlab code \verb+mkr_box+ provided to us by Philipp Hungerl\"ander.
All implementations were done in Matlab.

We performed test computations on a 2-d domain $\omega_o=\omega_c=\Omega=(-1,1)^2$, on a regular computational finite element grid consisting of $2\cdot N\cdot N$ triangles, with $N=32$. We first of all consider $\kappa=1$ (below we will also show results with $\kappa=100$) and the piecewise constant exact source function 
\begin{equation}\label{src_ex}
{\src}_{ex}(x,y)=-10+20\cdot {1\!\!{\rm I}}_{B}\,,
\end{equation}
where $B=\{(x,y)\in\R^2\, : \, (x+0.4)^2+(y+0.3)^2\leq0.04\}$
cf. Figure \ref{fig:exsol_spots}, and correspondingly set $\rho=10$. In order to avoid an inverse crime, we generated the synthetic data on a finer grid and, after projection of $u_{ex}$ onto the computational grid, we added normally distributed random noise of levels $\delta\in\{0.001, 0.01, 0.1\}$ to obtain synthetic data $\ydel$. In all our computations we chose $\tau=1.1$.

In all tests we start with the constant function with value zero for $\src_0$. Moreover, we always set $\tau=1.1$. 
\begin{figure}
\includegraphics[width=0.48\textwidth]{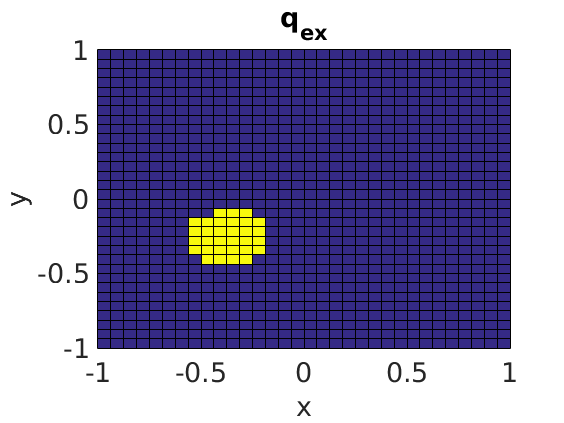}
\hspace*{0.01\textwidth}
\includegraphics[width=0.48\textwidth]{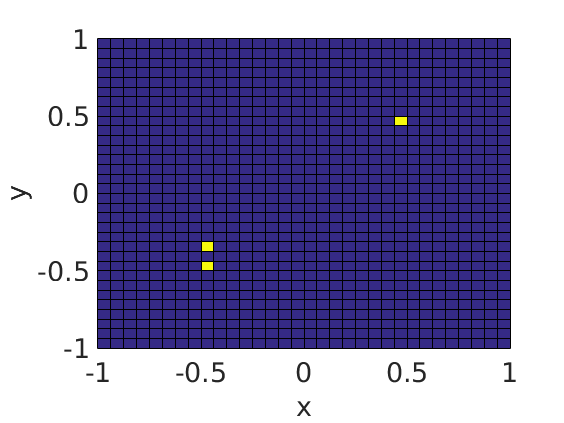}
\caption{
left: exact source ${\src}_{ex}$; 
right: locations of spots for testing weak * $L^\infty$ convergence
\label{fig:exsol_spots}}
\end{figure}
According to our convergence result Theorem \ref{th:conv} with $\cR=\|\cdot\|_{L^\infty(\Omega)}$, we can expect weak * convergence in $L^\infty(\Omega)$ here. Thus we computed the errors in certain spots within the two homogeneous regions and on their interface,
\[
\mbox{spot}_1=(0.5,0.5)\,, \quad
\mbox{spot}_2=(-0.4,-0.3)\,, \quad
\mbox{spot}_3=(-0.4,-0.5)\,, \quad
\]
cf. Figure \ref{fig:exsol_spots}, more precisely, on $\frac{1}{N}\times\frac{1}{N}$ squares located at these spots, corresponding to the piecewise constant $L^1$ functions with these supports in order to exemplarily test weak * $L^\infty$ convergence.
Additionally we computed $L^1$ errors.

Table \ref{tab_deltas} provides an illustration of convergence as $\delta$ decreases. For this purpose, we performed five runs on each noise level for each example and list the average errors. 

\begin{table}
\begin{tabular}
{|l||l|l|l|l||}
\hline
$\delta$&$\mbox{err}_{spot_1}$&$\mbox{err}_{spot_2}$&$\mbox{err}_{spot_3}$&$\mbox{err}_{L^1(\Omega)}$
\\ \hline
    0.1000&     0&4.0818&8.0043&0.0627
\\ \hline
    0.0667&0.1558&3.6454&7.8451&0.0541
\\ \hline
    0.0333&     0&3.0442&6.5726&0.0370
\\ \hline
    0.0100&     0&     0&3.9091&0.0188
\\ \hline
\end{tabular}
\\[2ex]
\caption{Convergence as $\delta\to0$: Averaged errors of five test runs with uniform noise
\label{tab_deltas}}
\end{table}

In Figures \ref{fig_conv} we plot the reconstructions for $\kappa=1$ and $\kappa=100$.
For $\kappa=1$, the noise levels $\delta\in\{0.1, 0.667,0.333,0.01\}$ correspond to a percentage of $p\in\{5.6, 18.5, 37.1, 55.6\}$ of the $L^2$ deviation of the exact state from the background state $u_0=-10^{1/3}$. In case of $\kappa=100$, where the background state is $u_0=-0.1^{1/3}$ the corresponding percentages are $p\in\{17.9,59.7,119.4,179.2\}$.
For an illustration of the noisy data as compared to the exact ones, see Figures \ref{fig_uexydel}, \ref{fig_uexydel_kappa100}. Indeed, the box constraints enable to cope with relatively large noise levels, even in the rather nonlinear regime with $\kappa=100$.
\begin{figure}
\includegraphics[width=0.43\textwidth]{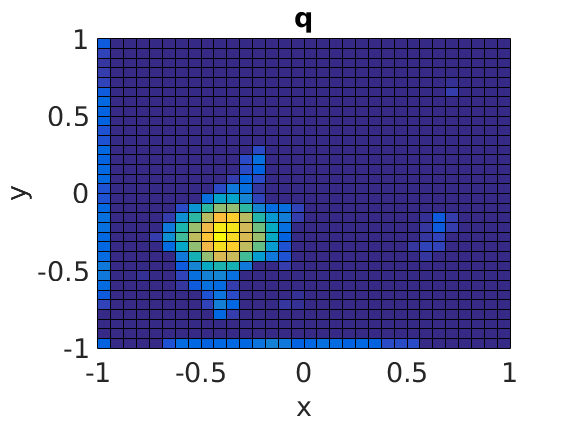}
\hspace*{0.01\textwidth}
\includegraphics[width=0.43\textwidth]{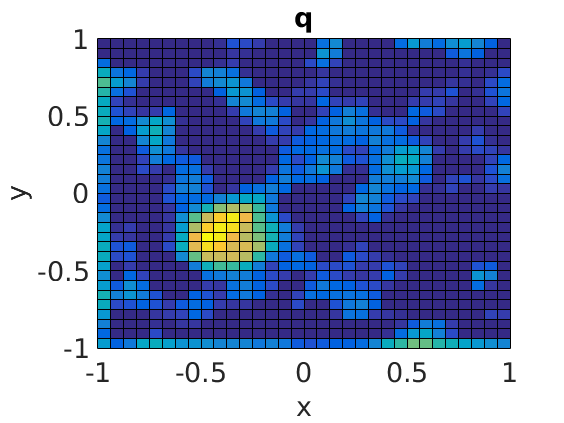}
\\
\includegraphics[width=0.43\textwidth]{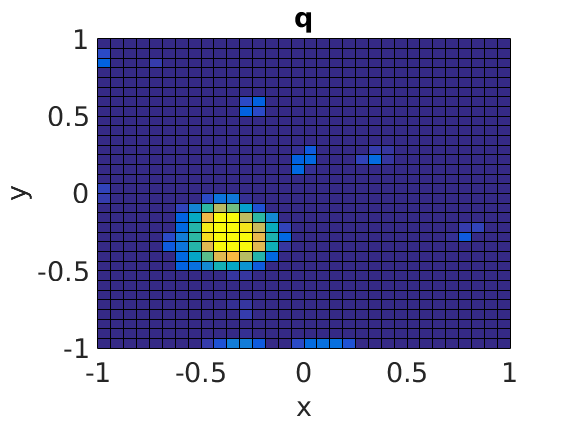}
\hspace*{0.01\textwidth}
\includegraphics[width=0.43\textwidth]{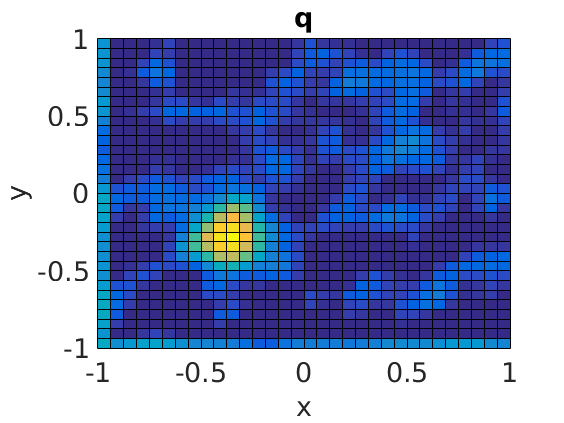}
\\
\includegraphics[width=0.43\textwidth]{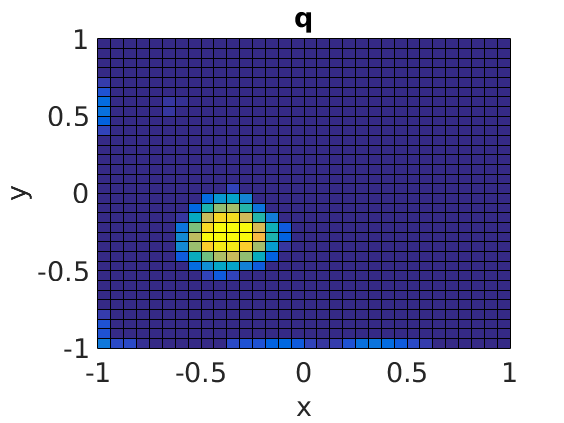}
\hspace*{0.01\textwidth}
\includegraphics[width=0.43\textwidth]{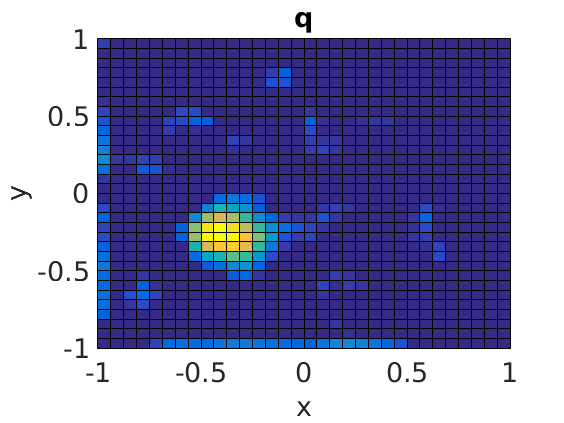}
\\
\includegraphics[width=0.43\textwidth]{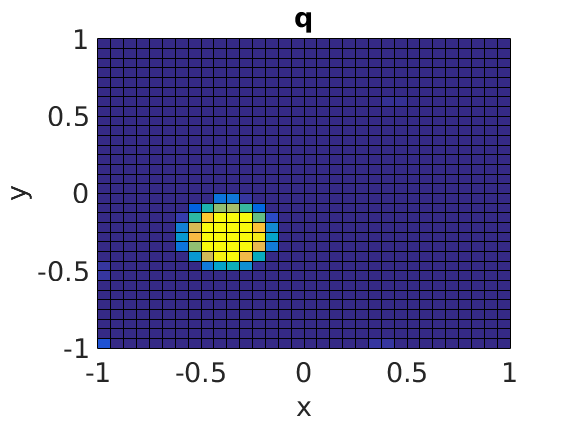}
\hspace*{0.01\textwidth}
\includegraphics[width=0.43\textwidth]{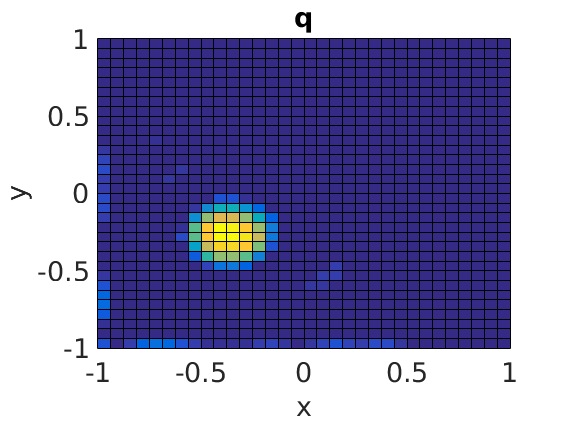}
\caption{
reconstructions from noisy data with $\delta\in\{0.1, 0.667,0.333,0.01\}$ (top to bottom) for $\kappa=1$ (left) and $\kappa=100$ (right)
\label{fig_conv}}
\end{figure}
\begin{figure}
\includegraphics[width=0.43\textwidth]{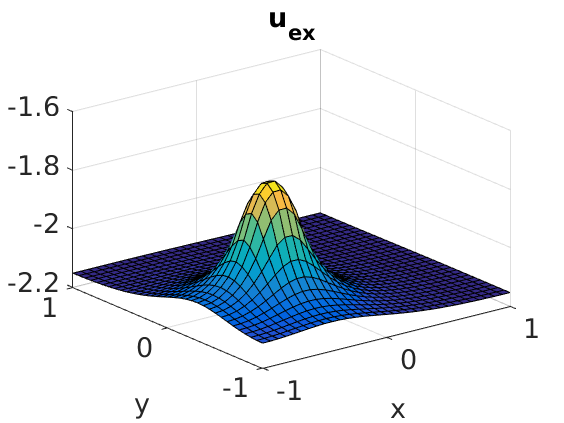}
\hspace*{0.01\textwidth}
\includegraphics[width=0.43\textwidth]{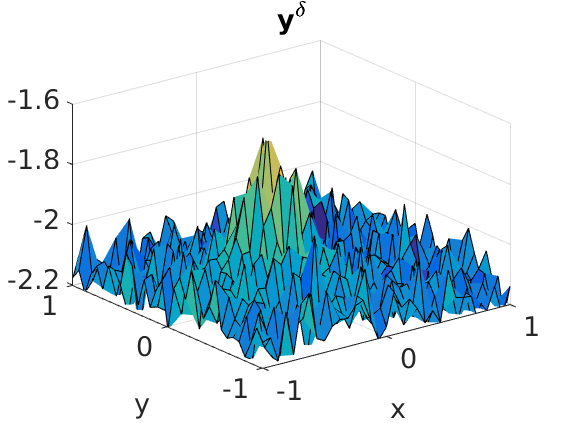}
\caption{
exact and noisy data ($\delta=0.1$) for $\kappa=1$
\label{fig_uexydel}}
\end{figure}
\begin{figure}
\includegraphics[width=0.43\textwidth]{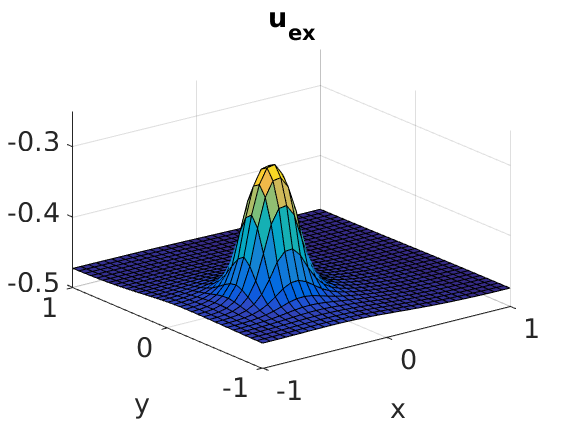}
\hspace*{0.01\textwidth}
\includegraphics[width=0.43\textwidth]{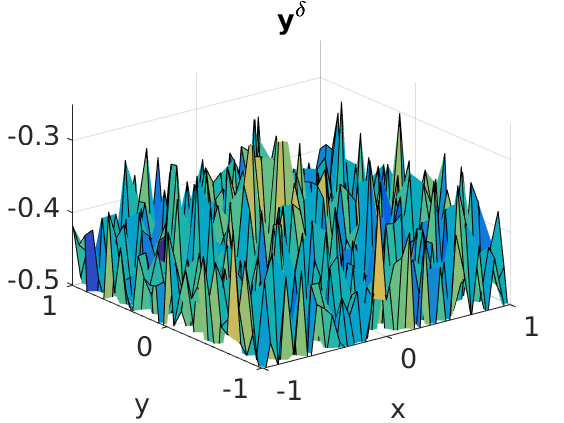}
\caption{
exact and noisy data ($\delta=0.1$) for $\kappa=100$
\label{fig_uexydel_kappa100}}
\end{figure}

\section{Conclusions and remarks}\label{sec:conclrem}
In this paper we have studied convergence of the Tikhonov type and the Ivanov type IRGNM with a stopping rule based on the discrepancy principle type. To the best of our knowledge, the Ivanov IRGNM method has not been studied so far and in both Ivanov and Tikhonov type IRGNM, convergence results without source conditions so far use stronger assumptions than the tangential cone condition used here.
We also consider discretized versions of the methods and provide discretization error bounds that still guarantee convergence. Moroever, we discuss goual oriented dual weighted residual error estimators that can be used in an adaptive discretization scheme for controlling these discretization error bounds.
An inverse source problem for a nonlinear elliptic boundary value problems illustrates our theoretical findings in the special situations of measure valued and $L^\infty$ sources.
We also provide some computational results with the IRGNM Ivanov method for the case of an $L^\infty$ source. Numerical implementations and test for a measure valued source, together with adaptive discretization is subject of ongoing work, based on the approaches from \cite{ClaKun,ClaKun2,ClaKunCasas,KKV14,KKV14b}.
Future research in this context will be concerend with convergence rates results for the IRGNM Ivanov method under source conditions.

\section*{Appendix}
The estimates in \eqref{abab} can be done by solving the following extremal value problems
\[
C_\gamma = \max_{x>0}{\phi(x)}\,, \quad C_\epsilon= \max_{x>0}{\Phi(x)}\,,
\]
where
\[
\phi (x):= ((1+x)^p-(1+\gamma)^{p-1})x^{-p} \mbox{ and } 
\Phi (x):= ((1-\epsilon)^{p-1}-(1-x)^p)x^{-p},
\]
since for any $\gamma,\epsilon \in (0,1)$,
\[
\phi(x)\leq C_\gamma \mbox{ and } \Phi(x)\leq C_\epsilon \mbox{ for all }x>0
\]
with $x:=b/a$, $a,b>0$ is equivalent to \eqref{abab}.

Solving for $C_\gamma$, we have 
\[
\phi ' (x) = px^{-(p+1)}((1+\gamma)^{p-1}-(1+x)^{p-1})\left\{\begin{array}{ll}
=0 &\Longleftrightarrow x=\gamma,\\
<0 &\mbox{ for } x>\gamma,\\
>0 &\mbox{ for } x<\gamma,
\end{array}\right.
\]
which means that
\begin{eqnarray*}
\max \phi(x)=\phi (\gamma)=\left(\frac{1+\gamma}{\gamma}\right)^{p-1},
\end{eqnarray*}
so defining $C_\gamma:=\left(\frac{1+\gamma}{\gamma}\right)^{p-1}$ and writing the resulting inequality in terms of $a$ and $b$ we have the desired formula.

The other formula in \eqref{abab} is derived analogously.

\section*{Acknowledgment}
The authors wish to thank Philipp Hun\-ger\-l\"ander, Alpen-Adria Universit\"at Klagenfurt, for providing us with the Matlab code based on the method from \cite{HungerlaenderRendl15}. 
Moreover, the authors gratefully acknowledge financial support by the Austrian Science Fund FWF under the grants I2271 ``Regularization and Discretization of Inverse Problems for PDEs in Banach Spaces'' and P30054 ``Solving Inverse Problems without Forward Operators'' as well as partial support by the Karl Popper Kolleg ``Modeling-Simulation-Optimization'', funded by the Alpen-Adria-Universit\" at Klagenfurt and by the Carin\-thian Economic Promotion Fund (KWF).\\

\end{document}